\documentclass[twoside]{article}
\usepackage{mor}

\received{}                              %
\revised{}                               %
\pubyear{0x}                             %
\pubmonth{Xxxxxxx}                       %
\volume{xx}                              %
\issue{x}                                %
\pages{xxx--xxx}                         %
\firstpage{0xxx}                         %
\DOI{10.1287/moor.xxxx.xxxx}             %
\startpagenumber{1}                      %
\newcommand{\vv}{\mathbf}

\title{%How Big Queues Occur in Multi-Server Systems
On Large Delays in Multi-Server Queues with Heavy Tails
}
\ShortTitle{Multi-Server Systems with Heavy Tails}
\ShortAuthors{S. Foss and D. Korshunov}
\NumberOfAuthors{2}
\FirstAuthor{Sergey Foss}
\FirstAuthorAddress{Department of Actuarial Mathematics and Statistics and the Maxwell
Institute for Mathematical Sciences, Heriot-Watt University,
Edinburgh EH14 4AS, Scotland and \\
Sobolev Institute of Mathematics,
Novosibirsk 630090, Russia
}
\FirstAuthorEmail{s.foss@hw.ac.uk}
\FirstAuthorURL{http://www.ma.hw.ac.uk/ams/people/pages/foss.php}
\SecondAuthor{Dmitry Korshunov}
\SecondAuthorAddress{Sobolev Institute of Mathematics,
Novosibirsk 630090, Russia}
\SecondAuthorEmail{korshunov@math.nsc.ru}
\SecondAuthorURL{http://www.math.nsc.ru/LBRT/v1/dima/dima.html}

\keywords{FCFS multi-server queue;
stationary waiting time;
heavy tails;
large deviations;
long tailed distribution;
subexponential distribution;
existence of moments}

\MSCcodes{Primary: 60K25, 90B22; Secondary: 60F10} % {\em See} \url{http://www.ams.org/msc}      }
\ORMScodes{Primary:  Queues Approximations} % {\em See} \url{http://mor.pubs.informs.org/ORSubject.pdf}
\begin{document}
\maketitle

\begin{abstract}
We present upper and lower bounds for
the tail distribution of the stationary waiting
time $D$ in the stable $GI/GI/s$ FCFS queue.
These bounds depend on the value of the
traffic load $\rho$ which is the ratio of mean service and
mean interarrival times.
For service times with
%SF
intermediate
%SF
regularly varying tail distribution
the bounds are exact up to a constant, and we are able
to establish a ``principle of $s-k$ big jumps'' in this case
(here $k$ is the integer part of $\rho$),
%SF
which gives the most probable way for the stationary waiting
time to be large
%SF
.

Another corollary of the bounds obtained is to
provide a new proof of necessity and sufficiency
of conditions for the existence of moments of the stationary
waiting time.
\end{abstract}

\normalsize

\section{Introduction and main results.}\label{intro}

We consider a {\it first-come-first-served}
multi-server system with $s$ identical servers.
Let $\tau$ be a typical interarrival time and
$\sigma$ a typical service time.
Independent identically distributed sequences
of interarrival times $\{\tau_n\}$ with mean
$a={\mathbb E}\tau$ and service times $\{\sigma_n\}$
with mean $b={\mathbb E}\sigma$ are assumed to be mutually independent.
We also assume throughout the paper that the distribution
of $\sigma$ has unbounded support, i.e.
$B(x):={\mathbb P}\{\sigma\le x\}<1$ for all $x$, and that the
system is {\it stable}, i.e. $\rho := b/a < s$.

There are two equivalent ways to describe
the dynamics of a multi-server system.
First, we may assume that customers form a single
queue in front of all servers, and that the first customer
in the queue moves immediately to a server which
becomes idle. Second, we may assume that customers
form $s$ individual queues (lines) -- one queue for
each server, service times of customers become
known upon their arrival, and each arriving
customer is directed to the line with a minimal
total workload (we also assume that queues are
numbered, and if there are more than one minimal
workloads, then a customer chooses the one with the
minimal number). In the rest of the paper,
we mostly follow the second description of the model.

For $n=1$, $2$, \ldots, let $\vv V_n=(V_{n1},\ldots,V_{ns})$
be the vector of residual workloads in lines $1$, \ldots, $s$
which are observed by the $n$-th customer
upon its arrival into the system.
The value of $D_n:=\min\{V_{nj},j\le s\}$ is the waiting time,
or the delay which customer $n$ experiences.
The $n$-th customer joins the $i_n$-th line. Then
$$
i_n=\min\{i:V_{ni}=D_n\}
$$
and
\begin{eqnarray*}
V_{n+1,i} &=& \left\{
\begin{array}{ll}
(V_{ni}+\sigma_n-\tau_{n+1})^+ &\mbox{ if } i=i_n,\\
(V_{ni}-\tau_{n+1})^+ &\mbox{ if } i\not=i_n.
\end{array}
\right.
\end{eqnarray*}

Let $R(\vv w)=(R_1(\vv w),\ldots,R_s(\vv w))$
be the operator on ${\mathbb R}^s$ which orders
the coordinates of $\vv w\in{\mathbb R}^s$
in the non-descending order, i.e.,
$R_1(\vv w)\le\cdots\le R_s(\vv w)$.
For $n=1$, $2$, \ldots, put $\vv W_n=R\vv V_n$.
Then $D_n=W_{n1}$ and the vectors $\{\vv W_n\}$
satisfy the %celebrated
Kiefer--Wolfowitz \cite{KW55} recursion:
\begin{equation}\label{KiW}
\vv W_{n+1} = R((W_{n1}+\sigma_n-\tau_{n+1})^+,
(W_{n2}-\tau_{n+1})^+,\ldots, (W_{ns}-\tau_{n+1})^+).
\end{equation}

In a stable system,
%The system is assumed to be {\it stable},
%i.e. $\rho:=b/a<s$.
%Under this assumption,
there exists a unique
stationary distribution for the Kiefer-Wolfowitz vectors
$\vv W_n$,  and the distribution of $\vv W_n$
converges to the stationary distribution
in the total variation norm, as $n\to\infty$.
In particular, the same holds for the $D_n$:
there exists a unique
distribution of the stationary waiting time (delay)
$D$, and the distribution of $D_n$ converges to that
of $D$ in the total variation norm.

In a single server queue ($s=1$),
the waiting times $D_n$
satisfy the Lindley recursion \cite{Lindley}:
$$
D_{n+1}=(D_n+\sigma_n-\tau_{n+1})^+.
$$
Recall that, given $D_1=0$,
$D_{n+1}$ coincides in distribution with
$\max(S_k,k\le n)$ where $S_0=0$ and
$S_n=\sum_{k=1}^n(\sigma_k-\tau_{k+1})$, for $n\ge 1$.
%SF The stationary delay $D$ in the single server queue
%SF has been investigated in great detail. In particular,
It is well known
(see, for example, \cite{P,Ver,APQ}) that
the tail of stationary waiting
time $D$ is related to the service time distribution tail
$\overline B(x)={\mathbb P}\{\sigma>x\}$ via the equivalence
\begin{eqnarray}\label{W.single}
{\mathbb P}\{D>x\} &\sim& \frac{\rho}{1-\rho}
\overline B_r(x) \quad\mbox{ as }x\to\infty,
\end{eqnarray}
provided the {\it subexponentiality} of the
{\it residual service time distribution} $B_r$ defined by its tail
\begin{eqnarray*}
\overline B_r (x) &:=& \frac{1}{b}
\int_x^\infty \overline B(y)\,dy,\ \ x>0
\end{eqnarray*}
is guaranteed.
Recall that a distribution $G$ on ${\mathbb R}^+$
is {\it subexponential},
%SF
 $G\in {\cal S}$,
%SF
if $\overline{G*G}(x)\sim2\overline G(x)$ as $x\to\infty$.

It is also well-known that, in a single server queue,
for any $\gamma>0$,
$D$ has a finite $\gamma$th moment, ${\mathbb E} D^{\gamma}<\infty$
if and only if
${\mathbb E}\sigma^{\gamma+1}<\infty$, see \cite{KW56}.
Equivalently, ${\mathbb E}D^\gamma<\infty$ if and only if
%the $\gamma$th moment of the residual
%time distribution $B_r$ is finite, i.e.
\begin{eqnarray*}
{\mathbb E} \sigma_{r,1}^\gamma &<& \infty
\end{eqnarray*}
where random variable $\sigma_{r,1}$ has distribution $B_r$.

Less is known about the stationary delay $D$ in
the multi-server queue. It is well understood
that the heaviness of the stationary waiting time tail
distribution
depends substantially on the load $\rho$ on the system
(see, for example, the conjecture on tail equivalence
by Whitt in \cite{W};
existence results for moments in \cite{S,SS,SV,SV2011};
asymptotic results for fluid queues fed by heavy-tailed
on-off flows in \cite{BMZ,BZ}).
More precisely, the tail distribution depends on $\rho$ via the value
of its integer part $k = [\rho]\in\{0,1,\ldots,s-1\}$.% for which $k\le\rho<k+1$.

For a $GI/GI/s$ system, a heuristic idea on a probable way for the large deviations
to occur may be described as follows.
Take $N=x\frac{k}{b-ka}$, for a very large $x$.
Let all service times
$\sigma_{n-N-s+k}$, \ldots, $\sigma_{n-N-1}$
be big enough, say $\sigma_{n-N-i}>x+Na$,
$i=1$, \ldots, $s-k$.
Then the other $k$ servers form an unstable
$GI/GI/k$ queue system, because the cumulative
drift of the corresponding workloads
approximately equals $b-ka>0$.
In time $N$ all workloads of these queues
will exceed level $x$ (again approximately).
In this way, at time $N$, all $s$ workloads
become greater than $x$ with probability
which is asymptotically not less than
$\overline B^{s-k}(x+Na)\approx
\overline B^{s-k}\bigl(x\frac{b}{b-ka}\bigr)
=\overline B^{s-k}\bigl(x\frac{\rho}{\rho-k}\bigr)$.
We use these heuristic arguments below in Section
\ref{lower.bound} to derive a lower bound.
We follow more precise calculations
to obtain a better lower bound
of order $\overline B_r^{s-k}\bigl(x\frac{\rho}{\rho-k}\bigr)$.

%SF
We recall now a few basic properties of heavy-tailed distributions and relations
between them. A distribution function $F$ is
\begin{itemize}
\item
{\it long-tailed}, $F\in {\cal L}$, if
$\overline{F}(x+1) \sim \overline{F}(x)$, as $x\to\infty$;
\item
{\it dominated varying}, $F\in {\cal D}$,
if $\overline{F}(2x)\ge c\overline{F}(x)$, for some $c>0$ and for all
$x$;
\item
{\it intermediate regularly varying}, $F\in {\cal IRV}$,
if
$$
\lim_{\varepsilon \downarrow 0} \liminf_{x\to\infty}
\overline{F}(x(1+\varepsilon ))/\overline{F}(x) =1;
$$
\item
{\it regularly varying},
$F\in {\cal RV}$, if $\overline{F}(x) = l(x)x^{-\alpha}$
for $x>0$ where $\alpha \ge 0$ is the {\it index} of regular variation
and $l(x)$ is a {\it slowly varying at infinity}
function, i.e. $l(cx)\sim l(x)$ as $x\to\infty$.
\end{itemize}
The following relations are known:
\begin{equation}\label{relations}
{\cal RV} \subset {\cal IRV} \subset {\cal L} \cap {\cal D} \subset {\cal S},
\end{equation}
see e.g. \cite{FKZ}, pp. 33 and 54.
%SF

In \cite{FK}, we treated the case $s=2$ in detail
and found the {\it exact asymptotics} for ${\mathbb P}\{D>x\}$.
We also described the {\it most probable way
for the occurrence of the large deviations}.
That means that, for the stationary waiting time to be large,
two large service times have to be large if $\rho <1$ and
$B_r$ is a subexponential distribution
(see
\cite[Theorem 1]{FK}) and one service time has to be
large if $1<\rho <2$ and if $B$ is
%SF
{\it long-tailed} and $B_r$ is
{\it intermediate
regularly varying}
%SF
(see \cite[Theorem 2]{FK}).
We also obtained a number of simple bounds. First,
Theorem 1 in \cite{FK} yields the following

\begin{theorem}\label{co.2.max}
Let $s=2$, $\rho<1$, and let the residual time
distribution $B_r$ be subexponential.
Then the tail of the stationary waiting time
satisfies the asymptotic relation, as $x\to\infty$,
\begin{eqnarray*}
{\mathbb P}\{D>x\} &\sim&
\frac{\rho^2}{2-\rho}\Bigl[(\overline B_r (x))^2
+\int_0^\infty \overline B_r(x+ya)\overline B(x+y(a-b))dy\Bigr].
\end{eqnarray*}
As a corollary, one can obtain the following
bounds for the stationary waiting time, % get that
%the tail of the stationary waiting time
%admits the following bounds,
as $x\to\infty$:
\begin{eqnarray*}
\Bigl(\frac{\rho^2(2+\rho)}{2(2-\rho)}+o(1)\Bigr)
\overline B_r^2 (x)
\le {\mathbb P}\{D>x\}
\le \Bigl(\frac{\rho^2}{2(1-\rho)}+o(1)\Bigr)
\overline B_r^2 (x).
\end{eqnarray*}
Another corollary is: if, in addition, the distribution
$B$ is regularly varying with index $\gamma>1$,
then, as $x\to\infty$:
\begin{eqnarray*}
{\mathbb P}\{D>x\} &\sim& c(\overline B_r(x))^2,
\end{eqnarray*}
where
\begin{eqnarray*}
c &=& \frac{\rho^2}{2-\rho}\Bigl[1+\frac{\rho}{\gamma-1}
\int_0^\infty\frac{dz}{(1+z)^{\gamma-1}(1+z(1-\rho))^\gamma}\Bigr].
\end{eqnarray*}
\end{theorem}

For the case $\rho>1$, we also proved in \cite{FK}
%the following

\begin{theorem}\label{th.2.min.lower}
Let $s=2$, $1<\rho<2$, and let both $B$ and $B_r$ be
subexponential distributions.
Then the tail of the stationary waiting time
satisfies the following inequalities:
\begin{eqnarray*}
\limsup_{x\to\infty}\frac{{\mathbb P}\{D>x\}}
{\overline{B}_r(2x)}
&\le& \frac{\rho}{2-\rho},
\end{eqnarray*}
and, for any fixed $\delta>0$,
\begin{eqnarray*}
\liminf_{x\to\infty}\frac{{\mathbb P}\{D>x\}}
{\overline B_r\bigl(\frac{\rho+\delta}{\rho-1}x\bigr)}
&\ge& \frac{\rho}{2-\rho}.
\end{eqnarray*}
If, in particular, $B$ is subexponential and $B_r$
is
%SF
intermediate
%SF
regularly varying, %SF at infinity,
then
\begin{eqnarray*}
{\mathbb P}\{D>x\} &\sim& \frac{\rho}{2-\rho}
\overline B_r\Bigl(\frac{\rho}{\rho-1}x\Bigr)
\quad\mbox{ as }x\to\infty.
\end{eqnarray*}
\end{theorem}

For an arbitrary $s\ge 2$ number of servers,
the best result on the existence of moments
was obtained in \cite[Theorem 4.1]{SV}
(here ${\cal L}_1^\gamma$ is a specific class of distributions
introduced in \cite{SV}):

\begin{theorem}\label{mom.SV}
Let $k<\rho<k+1$ for some $k\in\{0,1,\ldots,s-1\}$.
Then:

(i) If ${\mathbb E}\sigma^\gamma<\infty$ then
${\mathbb E}D^{(s-k)(\gamma-1)}<\infty$.

(ii) If in addition $\sigma$ is in the class ${\cal L}_1^\gamma$,
then ${\mathbb E}D^{(s-k)(\gamma-1)}<\infty$
implies ${\mathbb E}S^\gamma<\infty$.
\end{theorem}

In the present paper we introduce a condition which
is both necessary and sufficient for the finiteness
of ${\mathbb E}D^\gamma$. We present this condition in
``probabilistic terms''.

\begin{theorem}\label{exist.mom}
Let $\sigma_{r,1}$,
$\sigma_{r,2}$, \ldots\ be independent random variables
with common distribution $B_r$.
Let $k<\rho<k+1$ for some $k\in\{0,1,\ldots,s-1\}$.
For any $\gamma>0$, ${\mathbb E}D^\gamma$
is finite if and only if
\begin{equation}\label{min.fin}
{\mathbb E}(\min(\sigma_{r,1},\ldots,\sigma_{r,s-k}))^\gamma
<\infty,
\end{equation}
\end{theorem}
see Section \ref{proof} for the proof.
Actually, this result (which is sharper than Theorem \ref{mom.SV})
may be deduced from the results of \cite{SV},
but was not stated there. The corresponding proof
in \cite{SV} involves a comparison with the so-called
semi-cyclic service discipline. %SF As we understood this proof,
%SF it is questionable to obtain upper bounds
%SF for the tail distribution of $D$ on this way.
To the best of our knowledge, the latter approach does not allow
one to obtain upper bounds for the tail distribution
of $D$.
%SF

The main aim of the present paper is to introduce
a novel approach for constructing upper bounds for
the stationary waiting time in multi-server queues
(see Section \ref{sec.upper.bound} below).
%SF
This allows us to derive estimates for the tail probabilities of the distribution of
the stationary waiting time if the common distribution of service times is of
supexponential type, and, further, to establish the {\it principle of big jumps}
in a particular case of intermediate varying distributions.
Also, based on the new approach, we will obtain
%SFwhich allows to obtain
a direct proof of Theorem
\ref{exist.mom} (see Section \ref{proof}).
%SF, and
%SFalso to derive estimates
%SF(which may lead to the sharp asymptotics)
%SFfor the tail probabilities of the distribution
%SFof the stationary waiting time if the common distribution
%SFof service times is of subexponential type.

The most explicit bounds are obtained for the case
$\rho<1$. % and are presented in the following

\begin{theorem}\label{co.s.max}
Let $\rho=b/a<1$ and let the residual time
distribution $B_r$ be subexponential.
Then the tail distribution of the stationary
waiting time admits the following bounds:
\begin{eqnarray*}
\frac{\rho^s}{s!} \le
\liminf_{x\to\infty}
\frac{{\mathbb P}\{D>x\}}{\overline B_r^s(x)}
&\le& \limsup_{x\to\infty}
\frac{{\mathbb P}\{D>x\}}{\overline B_r^s (x)}
\le \Bigl(\frac{\rho}{1-\rho}\Bigr)^s.
\end{eqnarray*}
\end{theorem}

We present here the lower and upper bounds only.
As it was described in \cite{FK},
the only case where the tail asymptotics are available
with an explicit constant multiplier is the case
of regularly varying service time distribution.
%But even in the case of two stations
The corresponding calculations
%related to this constant
are rather involved
and deal with the law of large numbers
and a summation over a specific $s$-dimensional domain
with planar boundaries. These calculations have been
carried out
in \cite{FK} in the case of
$s=2$ servers.
%and invite other researchers to tackle the general case.

The proof of Theorem \ref{co.s.max} (see Section
\ref{k=0}) is based on a simple argument which
cannot be applied if $\rho>1$. For an arbitrary $\rho$,
we have the following result.

\begin{theorem}\label{th:subexp.bounds}
Let $k\in\{0,\ldots,s-1\}$ and $\delta>0$.
If $\rho>k$, then
\begin{eqnarray*}
{\mathbb P}\{D>x\}
&\ge& \frac{\rho^{s-k}+o(1)}{(s-k)!}
\overline{B_r}^{s-k}\Bigl(\frac{\rho+\delta}{\rho-k}x\Bigr)\ \mbox{ as }x\to\infty.
\end{eqnarray*}
If $\rho<k+1$  and if the residual service time distribution
$B_r$ is subexponential, then
\begin{eqnarray*}
{\mathbb P}\{D>x\}
&\le& \Bigl(\begin{array}{c}s\\k\end{array}\Bigr)
\Bigl(\frac{(k+1)\rho}{k+1-\rho}+o(1)\Bigr)^{s-k}
\overline B_r^{s-k}(x(1-\delta))\ \mbox{ as }x\to\infty.
\end{eqnarray*}
\end{theorem}

The lower bound follows from Theorem \ref{th.lower}
in Section \ref{lower.bound}. % where it is obtained
%without any restrictions on the service
%time distribution $B$.
The proof of the upper bound
may be found in Section \ref{sec.upper.bound}.
It is based on
%SF
results from
%SF
Section \ref{majorant},
where we present a novel construction of a
consistent majorant for $D_n$.

Note that the lower and the upper bounds in
Theorem \ref{th:subexp.bounds} are not necessarily
of the same order. In particular, if distribution
$B$ is of Weibull type then the ratio of the upper
and  the lower bounds tends to infinity, as $x$ increases.
In this case we do not have any ideas about how
correct/exact/sharp bounds would look like.
But if, in particular, the residual
 service time distribution belongs to the class ${\cal D} \cap
{\cal L}$,   %SFregularly varying,
%SF
then
%in the particular case of service time
%with a regularly varying distribution
%these bounds do coincide up to a multiplier constant.
these bounds differ by a multiplicative constant only.

Note that in Theorem \ref{th:subexp.bounds}
we require conditions on the residual distribution $B_r$
and not on the distribion $B$ itself. This is in line
with the key results on subexponentiality like (\ref{W.single}).

\begin{corollary}\label{regularly.case}
Let the residual service time distribution
%SF $\overline B(x)=l(x)/x^\alpha$, where $\alpha>1$ and
%$SF l(x)$ is a slowly varying at infinity function.
$B_r$ be long-tailed and dominated varying.
Let $k<\rho<k+1$ for some $k\in\{0,1,\ldots,s-1\}$.
Then there exist constants $c_1$ and $c_2$
such that, for all $x$,
\begin{equation}\label{rvst}
c_1\overline B_r^{s-k}(x) \le
{\mathbb P}\{D>x\} \le c_2\overline B_r^{s-k}(x).
\end{equation}
\end{corollary}

The result follows directly from Theorem \ref{th:subexp.bounds},
the last inclusion in (\ref{relations})
and the definition of the dominated variation.

%SF
%Indeed, by Karamata's Theorem,
%$\overline B_r(x)\sim l(x)/b(\alpha-1)x^{\alpha-1}$
%as $x\to\infty$.
%The bounds are due to Theorem \ref{th:subexp.bounds}
%and the property $\overline B_r(cx)\sim c^{1-\alpha}\overline B_r(x)$.
%SF

In the particular case of distributions of intermediate variation,
we will use \cite[Theorem 7]{BaF} to complement
 Corollary \ref{regularly.case} by establishing
the ``principle of $s-k$ big jumps'':
the main cause of the value of $D$ to be big is to have
$s-k$ big service times, see Section \ref{sec.upper.bound} for
the precise statement.
%SF

\section{Comparison of systems with different
inter-arrival times.}\label{comparison}

Here we present results which, in particular,
allow us to obtain lower and upper bounds for the stationary
delay in a
general $GI/GI/s$ system in terms of a simpler $D/GI/s$
system with deterministic
interarrival times. We use the following partial ordering:
for two vectors $\vv x = (x_1,\ldots ,x_s)$ and
$\vv y = (y_1,\ldots ,y_s)$, we write
$\vv x \le \vv y$ if $x_j\le y_j$
for all $j=1$, \ldots, $s$.

Consider two $GI/GI/s$ systems,
say $\widetilde{\vv V}$ and $\widehat{\vv V}$,
with service times $\sigma_n$
and with interarrival times
$\widetilde\tau_n$ and $\widehat\tau_n$ respectively.
Let $\widetilde D_n$ and $\widehat D_n$
be the corresponding waiting times in these systems. Let
$\xi_n=\widehat\tau_{n+1}-\widetilde\tau_{n+1}$.
We obtain an upper bound for delay $\widetilde D_n$
in terms of delay $\widehat D_n$
and the sequence $\xi_n$.

\begin{lemma}\label{lem:upper.bound}
For all $n\ge1$,
$\widetilde D_n \le \widehat D_n+M_{n-1}$,
where $M_0=0$ and $M_n=(M_{n-1}+\xi_n)^+$.
\end{lemma}

\begin{proof}
Put $\vv e_1=(1,0,\ldots,0)$ and $\vv 1=(1,\ldots,1)$.
It suffices to prove the inequality
\begin{eqnarray}\label{upper.via.M}
\widetilde{\vv W}_n &\le& \widehat{\vv W}_n+\vv 1M_{n-1}
\quad\mbox{ a.s.}
\end{eqnarray}
We proceed by induction. For $n=1$
we have $\vv 0\le\vv 0+\vv 1M_0$. Assume inequality
(\ref{upper.via.M}) to hold
for some $n$ and prove it for $n+1$.
We have
\begin{eqnarray*}
\widetilde{\vv W}_{n+1} &=&
R(\widetilde{\vv W}_n+\vv e_1\sigma_n
-\vv 1\widetilde\tau_{n+1})^+\\
&\le& R(\widehat{\vv W}_n+\vv 1M_{n-1}
+\vv e_1\sigma_n-\vv 1\widetilde\tau_{n+1})^+\\
&=& R(\widehat{\vv W}_n+\vv e_1\sigma_n
-\vv 1\widehat\tau_{n+1}
+\vv 1(M_{n-1}+\xi_n))^+.
\end{eqnarray*}
Since $(u+v)^+\le u^++v^+$,
\begin{eqnarray*}
\widetilde{\vv W}_{n+1}
&\le& R(\widehat{\vv W}_n+\vv e_1\sigma_n
-\vv 1\widehat\tau_{n+1})^++\vv 1(M_{n-1}+\xi_n)^+
\equiv \widehat{\vv W}_{n+1}+\vv 1M_n,
\end{eqnarray*}
and the proof of (\ref{upper.via.M}) is complete.
%Note that Lemma \ref{lem:upper.bound}
%is a sample-path result.
\end{proof}

The following corollary
will be used to obtain lower bounds.
It is similar to Lemma 2 in \cite{FK}.

\begin{corollary}\label{cor:lower.bound.D}
Let $\vv W_n'$ be a stable $s$-server
queue system with the same service times $\sigma_n$
as in $\vv W_n$ and with the constant interarrival times $a'$.
If $a'>a={\mathbb E}\tau$, then, for any $\varepsilon>0$,
there exists $x_0$ such that
\begin{eqnarray*}
{\mathbb P}\{D>x\} &\ge&
(1-\varepsilon){\mathbb P}\{D'>x+x_0\}\ \mbox{ for all }x.
\end{eqnarray*}
One can take $x_0$ such that
$$
{\mathbb P}\Bigl\{\sup_{n\ge 0}
\sum_{i=1}^n(\tau_i-a')\le x_0\Bigr\}\ge1-\varepsilon.
$$
\end{corollary}

\begin{proof}
Take $\widetilde\tau_n=a'$ and $\widehat\tau_n=\tau_n$
in Lemma \ref{lem:upper.bound},
then $\xi_n=\tau_n-a'$.
A weak limit, $M$, of the sequence $M_n$ exists (since
${\mathbb E}\xi_1=a-a'<0$)
%As was already mentioned,
%the following stochastic equality holds:
and has the same distribution as
\begin{eqnarray*}
M &=_{\rm st}& \max\{0,\ \xi_1,\
\xi_1+\xi_2,\ \ldots,\ \xi_1+\cdots+\xi_n,\ \ldots\}.
\end{eqnarray*}
By Lemma \ref{lem:upper.bound}, $D_n'\le D_n+M_{n-1}$.
Hence, $D_n\ge D'_n-M_{n-1}$
Since $D'_n$ does not depend on $\tau$'s,
$D'_n$ and $M_{n-1}$ are independent. Therefore,
\begin{eqnarray*}
{\mathbb P}\{D_n>x\} &\ge&
{\mathbb P}\{M_{n-1}\le x_0\}{\mathbb P}\{D'_n>x+x_0\}.
\end{eqnarray*}
Letting $n$ go to infinity,
we obtain the desired bound.
\end{proof}

\section{The case $\rho<1$, proof of Theorem \ref{co.s.max}.}
\label{k=0}

The lower bound in Theorem \ref{co.s.max} follows
from Lemma \ref{l.lower.k=0} below which also
generalises Theorem \ref{th.lower}
(see Section \ref{lower.bound}) in the case $k=0$.

\begin{lemma}\label{l.lower.k=0}
Let $\rho>0$. Then, for any function
$h(x)\to\infty$ as $x\to\infty$,
$$
{\mathbb P}\{D>x\} \ge \frac{\rho^s+o(1)}{s!}
\overline{B_r}^s(x+h(x)).
$$
In particular, if the residual time
distribution $B_r$ is long-tailed (that is,
$\overline B_r(x+1) \sim \overline B_r(x)$
as $x\to\infty$), then
$$
{\mathbb P}\{D>x\} \ge \frac{\rho^s+o(1)}{s!}
\overline{B_r}^s(x)\ \mbox{ as }x\to\infty.
$$
\end{lemma}

We start with an auxiliary result.

\begin{lemma}\label{q.k}
Let $\{q_i\}_{i\ge 1}$ be a non-increasing
sequence of positive numbers. Then, for any $s\ge 1$,
\begin{eqnarray*}
\sum_{1\le i_1<\ldots<i_s}
q_{i_1}\cdot\ldots\cdot q_{i_s}
&\ge& \frac{1}{s!}(q_s+q_{s+1}+\ldots)^s.
\end{eqnarray*}
\end{lemma}

\begin{proof}
If $1\le i_1<\ldots<i_s$ then
$s\le i_1+(s-1)\le i_2+(s-2)\le\ldots\le i_{s-1}+1\le i_s$ and
$$
q_{i_1}\cdot q_{i_2}\cdot\ldots\cdot q_{i_s}
\ge q_{i_1+s-1}\cdot q_{i_2+s-2}\ldots\cdot q_{i_s},
$$
because $\{q_i\}$ is a non-increasing sequence.
Thus,
\begin{eqnarray*}
\sum_{1\le i_1<\ldots<i_s}
q_{i_1}\cdot\ldots\cdot q_{i_s}
&\ge& \sum_{s\le i_1\le \ldots\le i_s}
q_{i_1}\cdot\ldots\cdot q_{i_s}\\
&\ge& \frac{1}{s!}\sum_{i_1,\ldots,i_s\ge s}
q_{i_1}\cdot\ldots\cdot q_{i_s},
\end{eqnarray*}
which yields the conclusion of the lemma.
\end{proof}

\begin{proofof}{Lemma \ref{l.lower.k=0}}
Our estimation is based on calculations involving $s$
big jumps.
This technique was already used in \cite{FK} in
the case $s=2$, where a lower bound
(which is better than the one presented in
Lemma \ref{l.lower.k=0})
was obtained under the extra condition that $B_r$ is long-tailed.
The bound in \cite{FK} is exact in the sense that
it provides the right asymptotics under further
assumptions.

Following Corollary \ref{cor:lower.bound.D},
define the auxiliary $s$-server system $\vv W'_n$
having the same service times $\sigma_n$ and
constant interarrival times $a'$, $a'>a$.
For $\vv i=(i_1,\ldots,i_s)$,
$1\le i_1<\ldots<i_s<n$, define  events
$A_n(\vv i)$ and $C_n(\vv i)$ as
\begin{eqnarray*}
A_n(\vv i) &=& \{\sigma_{i_1}>x+(n-i_1)a', \ldots,
\sigma_{i_s}>x+(n-i_s)a'\}
\end{eqnarray*}
and
\begin{eqnarray*}
C_n(\vv i)
&=& \bigcap_{i\le n,i\ne i_1,\ldots,i_s}
\{\sigma_i\le x+(n-i)a'\}.
\end{eqnarray*}
Since the mean ${\mathbb E}\sigma$ exists, we have that,
uniformly in $n$ and $\vv i$,
\begin{eqnarray*}
{\mathbb P}\{C_n(\vv i)\}
=1-{\mathbb P}\{\overline{C_n(\vv i)}\}
&\ge& 1-\sum_{i=0}^\infty {\mathbb P}\{\sigma_1>x+ia'\}
\to 1 \quad\mbox{ as }x\to\infty.
\end{eqnarray*}
For each vector $\vv i$, events $A_{n}(\vv i)$ and
$C_{n}(\vv i)$ are independent. Further,
events $A_n(\vv i)\cap C_n(\vv i)$
are disjoint for distinct vectors $\vv i$.
These observations together yield
\begin{eqnarray}\label{union.of.An}
{\mathbb P}\Bigl\{\bigcup_{\vv i}A_n(\vv i)
\cap C_n(\vv i)\Bigr\}
&=& \sum_{\vv i}{\mathbb P}\{A_n(\vv i)\}{\mathbb P}\{C_n(\vv i)\}
\ge (1-o(1))\sum_{\vv i}{\mathbb P}\{A_n(\vv i)\}
\end{eqnarray}
as $x\to\infty$, uniformly in $n$.
The event $A_n(\vv i)$ implies that $D'_n>x$. Therefore,
\begin{eqnarray*}
{\mathbb P}\{D'_n>x\} &\ge& (1-o(1))
\sum_{\vv i}{\mathbb P}\{A_n(\vv i)\}
\end{eqnarray*}
as $x\to\infty$, uniformly in $n$. We now prove that
\begin{eqnarray}\label{Sigma.k=0}
\lim_{n\to\infty}
\sum_{\vv i} {\mathbb P}\{A_n(\vv i)\}
&\ge& \frac{(b/a')^s}{s!}\overline{B_r}^s(x+sa').
\end{eqnarray}
Indeed, by the independence of the $\sigma$'s,
\begin{eqnarray*}
\sum_{1\le i_1<\ldots<i_s<n}
{\mathbb P}\{A_n(\vv i)\}
&=& \sum_{1\le i_1<\ldots,i_s<n}
\overline B(x+(n-i_1)a')
\cdot\ldots\cdot\overline B(x+(n-i_s)a'),
\end{eqnarray*}
and the left side of (\ref{Sigma.k=0}) equals
\begin{eqnarray*}
\sum_{1\le i_1<\ldots<i_s}
\overline B(x+i_1a')\cdot\ldots\cdot\overline B(x+i_sa').
\end{eqnarray*}
By Lemma \ref{q.k} with $q_i=\overline B(x+ia')$,
the latter sum is not smaller than
\begin{eqnarray*}
\frac{1}{s!} \Bigl(\sum_{j=s}^\infty
\overline B(x+ja')\Bigr)^s.
\end{eqnarray*}
Since the tail probability is a non-increasing function,
\begin{eqnarray*}
\sum_{j=s}^\infty\overline B(x+ja')
\ge \frac{1}{a'} \int_{sa'}^\infty\overline B(x+z)dz
=\rho\overline B_r(x+sa').
\end{eqnarray*}
Combining altogether, we conclude (\ref{Sigma.k=0}).
Then by Corollary \ref{cor:lower.bound.D},
for every $\varepsilon>0$ there exists $x_0$ such that
\begin{eqnarray*}
{\mathbb P}\{D>x\} &\ge&
(1-\varepsilon){\mathbb P}\{D'>x+x_0\}\\
&\ge& (1-\varepsilon-o(1))\frac{(b/a')^s}{s!}
\overline{B_r}^s(x+sa'+x_0).
\end{eqnarray*}
By the arbitrary choice of $a'>a$ and $\varepsilon>0$,
the proof of Lemma \ref{l.lower.k=0} is complete.
\end{proofof}

\begin{proofof}{the upper bound in Theorem \ref{co.s.max}}
We start with the case of deterministic $\tau$,
i.e., $\tau_n\equiv a$. We follow the lines from \cite{FK}
where, for $\rho<1$ (that is for $b<a$),
the following simple majorant was introduced.

Let $\sigma_{ni}$, $n\ge1$, $i\le s$, be
independent random variables with common distribution $B$.
Consider $s$ auxiliary $D/GI/1$
queueing systems which work in parallel:
at every time instant $T_n = na$, $n=1$, 2, \ldots,
a batch of $s$ customers arrives,
one customer per each queue.
%arrives in the $i$st queue,
%$i=1$, \ldots, $s$.
Service times in queue $i$ are equal to $\sigma_{ni}$.
Denote by $U_{ni}$, $i=1$, \ldots, $s$, the waiting
times in the $i$th queue, $U_{n+1,i}= (U_{ni}+\sigma_{ni}-a)^+$,
and let $U_{1i}=0$.
Since the arrival process is deterministic and service
times are independent, vector $(U_{n1},\ldots , U_{ns})$
has independent identically distributed coordinates,
and, as $n\to\infty$, its weak limit
$(U_{1},\ldots,U_{s})$ exists (since ${\mathbb E}\sigma<a$)
and contains independent identically distributed coordinates too.
Here $U_i$ is the stationary waiting time in
the $i$th auxiliary queue.

Now we introduce a coupling of $s$ single-server systems
and of the $s$-server system $D/GI/s$. Namely, we
determine the service times $\sigma_n$ in the original
$D/GI/s$ system by induction.
Start with $\sigma_1=\sigma_{1,1}$.
Assume that $\sigma_1$, \ldots, $\sigma_{n-1}$
have been already defined. Then the delay vectors
$\vv V_1$, \ldots, $\vv V_n$ are defined too, and
we know the number $i_n=\min\{i:V_{ni}=D_n\}$.
Then let $\sigma_n=\sigma_{n,i_n}$.

By monotonicity, $D_n\le \min\{U_{ni},\ i\le s\}$
with probability 1. Hence,
\begin{eqnarray}\label{D.min}
D &\le& \min\{U_i,\ i\le s\}.
\end{eqnarray}
Due to independence,
\begin{eqnarray*}
{\mathbb P}\{D>x\} &\le& {\mathbb P}^s\{U_1>x\},
\end{eqnarray*}
%for the case $\rho<1$,
and we can apply known results for the single
server queue: %Under subexponentiality of $B_r$,
%we have
from (\ref{W.single}),
\begin{eqnarray*}
{\mathbb P}\{U_1>x\} &\sim&
\frac{\rho}{1-\rho}\overline B_r(x),
\end{eqnarray*}
which gives us the upper bound
in Theorem \ref{co.s.max} if interarrival times are deterministic.
 Now the proof in the general case follows
from \cite[Lemma 1]{FK}.
\end{proofof}

\section{Auxiliary results.}\label{lln}

In this Section we collect a number
%In the proofs we need the a number
of auxiliary facts related to monotonicity and
to
the strong law of large numbers for
unstable multi-server systems.
The results seem not to be new, so we
provide only short sketches of proofs for self-containedness.

Let $\vv W_n$ be a sequence satisfying
the Kiefer-Wolfowitz recursion (\ref{KiW}),
with initial value $\vv W_1 \ge 0$.

\begin{lemma}\label{monotonicity}
(1)
For any $n$, $\vv W_n$ is a non-decreasing function
of the initial value and of service times and a non-increasing function of interarrival
times. This means that if $\widetilde{\vv W}_n$
is a sequence satisfying another Kiefer-Wolfowitz recursion  with initial
value $\widetilde{\vv W}_1$ and with interarrival times $\{\widetilde{\tau}_n\}$ and service
times $\{\widetilde{\sigma}_n\}$ and if
$\vv W_1 \le \widetilde{\vv W}_1$ (coordinate-wise), $\sigma_j\le \widetilde{\sigma}_j$,
and $\tau_j \ge \widetilde{\tau}_j$, for $j=1,\ldots ,n-1$,
then
$\vv W_n \le \widetilde{\vv W}_n$.\\
(2)
For any $n\ge 2$, the difference $\sum_{i=1}^s (W_{ni}
-W_{n-1,i})$ is a non-increasing function of the initial
value $\vv W_1$: if ${\vv W}_1 \le \widetilde{\vv W}_1$,
then $\sum_{i=1}^s (W_{ni}
-W_{n-1,i}) \ge \sum_{i=1}^s (\widetilde{W}_{ni}
-\widetilde{W}_{n-1,i}).$
\end{lemma}
The first monotonicity property holds because both operators $R$ and
$\max (0,\cdot )$ are monotone. The second property follows
since function $(x+y)^+-x$ is non-increasing in $x$, for any
fixed $y$.

\begin{lemma}\label{slln.for.unstable}
Let $b>sa$, so the $s$-server system with workload vectors ${\vv W}_n$ is unstable.
Then,
\begin{equation}\label{SLLN1}
\frac{W_{n1}}{n}\to \frac{b-sa}{s}
\quad \mbox{and}\quad
\frac{W_{ns}}{n}\to \frac{b-sa}{s}
\quad\mbox{ as }n\to\infty,
\end{equation}
both with probability $1$ and in mean.
\end{lemma}

\begin{proof}
Note that, for any $n=1,2,\ldots$,
\begin{equation}\label{ind1}
W_{n+1,s}-W_{n+1,1} \le \max (W_{1,s}-W_{1,1}, \sigma_1,\ldots , \sigma_n).
\end{equation}
Indeed, if $W_{ns}-W_{n1}>\sigma_n$, then $W_{n+1,s}-W_{n+1,1}\le W_{ns}-W_{n1}$,
and if $W_{ns}-W_{n1}\le \sigma_n$, then $W_{n+1,s}-W_{n+1,1}\le \sigma_n$, so the
induction argument completes the proof of (\ref{ind1}). Next,
\begin{equation}\label{ind2}
\max (W_{1s}-W_{11}, \sigma_1,\ldots , \sigma_n )/n \to 0  \quad \mbox{a.s.}
\end{equation}
because $(W_{1s}-W_{11})/n\to0$ and, since ${\vv E}\sigma$ is
finite, events $\{ \sigma_k/k > \varepsilon \}$ occur only finitely often,
for any $\varepsilon >0$.

Further,
$$
\frac{1}{n}\sum_{i=1}^s W_{ni} \ge \frac{1}{n}\sum_{j=1}^{n-1} (\sigma_j-s\tau_{j+1})
\to b-sa >0 \quad \mbox{a.s.,}
$$
so $\liminf_{n\to\infty}W_{ns}/n \ge (b-sa)/s$, and, from (\ref{ind1})-(\ref{ind2}),
there exists an a.s. finite random variable $\nu$ such that $W_{n1}>0$, for all $n\ge\nu$.
So, for $n\ge\nu$,
\begin{equation}\label{ind3}
\frac{1}{n}\sum_{i=1}^s W_{ni}
=\frac{1}{n}\sum_{i=1}^s W_{\nu i} + \frac{1}{n} \sum_{j=\nu}^n (\sigma_j-s\tau_{j+1})
\to b-sa \quad \mbox{a.s.,}
\end{equation}
and (\ref{ind1})-(\ref{ind3}) lead to convergence a.s. in (\ref{SLLN1}).
Finally, since  $0\le W_{ns}/n \le W_{1s}/n + \sum_{j=1}^{n-1} \sigma_j/n$
and since random variables $\sum_{j=1}^{n-1} \sigma_j/n$ are uniformly integrable,
convergence in mean also follows.
\end{proof}

\begin{lemma}\label{slln.for.s-1}
Assume $b>(s-1)a$. For any $\varepsilon>0$,
there exist $A<\infty$ and an integer $d\ge 1$ such that,
for any initial value ${\vv W}_1$ with $W_{1s}\ge A$,
$$
{\mathbb E}\{W_{1+d,1}+\ldots+W_{1+d,s}
-W_{11}-\ldots-W_{1s} \}  %\mid W_{1s}>A\}
\le d(b-sa+\varepsilon).
$$
\end{lemma}

\begin{proof}
By property (2) of Lemma \ref{monotonicity}, it is enough to prove the
result for initial value $W_{11}=\ldots =W_{s-1,1}=0$, $W_{s1}=A$ only.

Choose $C$ such that ${\mathbb E} \min (\tau, C) \ge a-\varepsilon /2$.
By property (1) of Lemma \ref{monotonicity}, we may prove the lemma
with interarrival times
$\min (\tau_j ,C)$  in place of $\tau_j$.

Consider an auxiliary unstable $GI/GI/(s-1)$ queue $\widehat{\vv W}_n$
with initial zero value
and, by applying the previous lemma, find $d$ such that
${\mathbb E} \sum_{i=1}^{s-1}\widehat{W}_{1+d,i} \le d(b-(s-1)a+\varepsilon/2)$.
Then return to the $s$-server queue and take $A=(d+1)C$.
We will prove that
\begin{equation}\label{L43}
\sum_{i=1}^{s}W_{1+d,i} = \sum_{i=1}^{s-1}\widehat{W}_{1+d,i} +
A - \sum_{j=1}^{d} \min (\tau_{j},C) \quad \mbox{a.s.,}
\end{equation}
then the result will follow.

Consider vectors ${\vv V}_n$ and numbers $i_n$ as in the Introduction, with initial
values $V_{1,1}=\ldots =V_{1,s-1}=0$ and $V_{1,s}=A$. Note that $V_{n,s}\ge A-(n-1)C>0$,
for all $n=1,2,\ldots,d+1$.

Let $\mu = \min ( d+1, \min \{n\ge 1 \ : \ i_n=s \} ).$ Then
$R(V_{\mu,1},\ldots,V_{\mu,s-1})=(\widehat{W}_{\mu,1}\ldots,\widehat{W}_{\mu,s-1})$
and
\begin{equation}\label{LL43}
\sum_{i=1}^s W_{\mu,i} = \sum_{i=1}^s V_{\mu,i}=
\sum_{i=1}^{s-1} V_{\mu,i}+V_{\mu,s} = \sum_{i=1}^{s-1}\widehat{W}_{\mu,i}+A-\sum_{j=1}^{\mu -1}
\min (\tau_j,C).
\end{equation}
This ends the proof of (\ref{L43}) if $\mu = d+1$.

In the case $\mu <d+1$, we may conclude that
$$
0<A-(\mu -1)C \le V_{\mu,s} = W_{\mu,1}
$$
and, therefore, $W_{n,i}>0$ and $\widehat{W}_{n,i}>0$, for all $\mu \le n \le d+1$ and $i=1,\ldots,s$.
Then, from (\ref{LL43}),
\begin{eqnarray*}
\sum_{i=1}^s W_{d+1,i}
&=&
\sum_{i=1}^s W_{\mu,i}+\sum_{j=\mu}^d (\sigma_j -s\min (\tau_j,C))\\
&=&
\sum_{i=1}^{s-1}\widehat{W}_{\mu,i} +\sum_{j=\mu}^d (\sigma_j- (s-1)\min (\tau_j,C))
+ A-\sum_{j=1}^d \min (\tau_j,C)
\end{eqnarray*}
which coincides again with the right side of (\ref{L43}).

\end{proof}

\section{Lower Bound.}\label{lower.bound}

The following result holds without any restrictions
on the service time
distribution $B$ %the following result holds
(a similar result was formulated
and proved in \cite[Theorem 3.1]{SV}).

\begin{theorem}\label{th.lower}
Let $k\in\{0,1,\ldots,s-1\}$ be such that $\rho>k$.
Then, for any fixed $\delta>0$,
$$
{\mathbb P}\{D>x\} \ge \frac{\rho^{s-k}+o(1)}{(s-k)!}
\overline{B_r}^{s-k}\Bigl(\frac{\rho+\delta}{\rho-k}x\Bigr)
\quad\mbox{as }x\to\infty.
$$
\end{theorem}

\begin{proof}
We exploit the technique of $s-k$ big jumps.
Following Corollary \ref{cor:lower.bound.D},
we consider only deterministic
interarrival times, $\tau\equiv a$.

The case $k=0$ was considered in Lemma \ref{l.lower.k=0}.
So  now let  $k\ge 1$.
Let $\widetilde{\vv W}_n=(\widetilde W_{n1},\ldots,\widetilde W_{nk})$
be the residual workload vector in the $GI/GI/k$ system
with the same interarrival and service times as in the original
system,
and with $k$ servers.
Since $\rho>k$, the $k$-server system is unstable.
Hence, by Lemma \ref{slln.for.unstable},
both the minimal coordinate $\widetilde W_{n1}$
and the maximal coordinate $\widetilde W_{nk}$
drift to infinity as $n\to\infty$ with probability 1,
with the same rate
$(b-ka)/k$. Then
\begin{eqnarray*}
{\mathbb P}\Bigl\{\widetilde W_{N1}>N\Bigl(\frac{b-ka}{k}-\delta\Bigr),
\ \widetilde W_{ik}\le N\Bigl(\frac{b-ka}{k}+\delta\Bigr)
\mbox{ for all }i\le N\Bigr\} &\to& 1
\ \mbox{ as }N\to\infty.
\end{eqnarray*}

If we assume that there are initially big workloads at $s-k$ servers
while the $k$ other queues are empty,
then, with high probability, the $k$ smallest workloads
evolve like the $k$-server system with workloads $\widetilde{\vv  W}_n$,
for a long while. This observation implies that
\begin{eqnarray*}
\lefteqn{{\mathbb P}\Bigl\{
W_{N1}>N\Bigl(\frac{b-ka}{k}-\delta\Bigr),
\ W_{ik}\le N\Bigl(\frac{b-ka}{k}+\delta\Bigr),
W_{i,k+1}>N\Bigl(\frac{b-ka}{k}+\delta\Bigr)
\mbox{ for all }i\le N\Big|}\\
&&\hspace{40mm}
W_{1k}=0, W_{1,k+1}>
N\Bigl(\frac{b-ka}{k}+\delta\Bigr)+Na\Bigr\}
\to 1\ \mbox{ as }N\to\infty.
\end{eqnarray*}
Take $c$ such that
\begin{eqnarray}\label{def.c}
\frac{b-ka}{k}+\delta
&\le& (1+c\delta)
\Bigl(\frac{b-ka}{k}-\delta\Bigr)
\end{eqnarray}
for all sufficiently small $\delta>0$, and let
\begin{eqnarray}\label{def.N}
x=N\Bigl(\frac{b-ka}{k}-\delta\Bigr).
\end{eqnarray}
Then
\begin{eqnarray*}
{\mathbb P}\{D_N>x\mid W_{1k}=0,
W_{1,k+1}>x(1+c\delta)+Na\}
&\to& 1\ \mbox{ as }x\to\infty.
\end{eqnarray*}
By the monotonicity of the $s$-server
queueing system in its initial state (see Lemma \ref{monotonicity}), we obtain
\begin{eqnarray}\label{slln.double}
{\mathbb P}\{D_N>x\mid W_{1,k+1}>x(1+c\delta)+Na\}
&\to& 1\ \mbox{ as }x\to\infty.
\end{eqnarray}

For $\vv i=(i_1,\ldots,i_{s-k})$,
$1\le i_1<\ldots<i_{s-k}\le n$, define the events
$A_n(\vv i)$ %and $C_n(\vv i)$
as
\begin{eqnarray*}
A_n(\vv i) &=& \{\sigma_{i_1}>y+(n-i_1)a, \ldots,
\sigma_{i_{s-k}}>y+(n-i_{s-k})a\}.
\end{eqnarray*}
Again like in (\ref{union.of.An}) we have
\begin{eqnarray}\label{union.An.N}
{\mathbb P}\Bigl\{\bigcup_{\vv i:i_{s-k}<n-N}
A_n(\vv i)\Bigr\}
&\ge& (1-o(1))\sum_{\vv i:i_{s-k}<n-N}
{\mathbb P}\{A_n(\vv i)\}
\end{eqnarray}
as $y\to\infty$, uniformly in $n$ and $N$.
We prove now that
\begin{eqnarray}\label{Sigma}
\lim_{n\to\infty}
\sum_{\vv i:i_{s-k}<n-N} {\mathbb P}\{A_n(\vv i)\}
&\ge& \frac{\rho^{s-k}}{(s-k)!}
\overline{B_r}^{s-k}(y+(N+s-k)a).
\end{eqnarray}
Indeed, by the independence of the $\sigma$'s,
\begin{eqnarray*}
\sum_{\vv i:i_{s-k}<n-N} {\mathbb P}\{A_n(\vv i)\}
&=& \sum_{\vv i:i_{s-k}<n-N} \overline B(y+(n-i_1)a)
\cdot\ldots\cdot\overline B(y+(n-i_{s-k})a)\\
&=& \sum_{N<i_1<\ldots<i_{s-k}\le n-1}
\overline B(y+i_1a)\cdot\ldots\cdot\overline B(y+i_{s-k}a).
\end{eqnarray*}
Hence, the left side of (\ref{Sigma}) equals
\begin{eqnarray*}
\sum_{N<i_1<\ldots<i_{s-k}}
\overline B(y+i_1a)\cdot\ldots\cdot\overline B(y+i_{s-k}a)
&\ge& \sum_{1\le i_1<\ldots<i_{s-k}}
\overline B(y+Na+i_1a)\cdot\ldots\cdot
\overline B(y+Na+i_{s-k}a).
\end{eqnarray*}
By Lemma \ref{q.k} with
$q_i=\overline B(y+Na+ia)$, the sum on the right
is not less than
\begin{eqnarray*}
\frac{1}{(s-k)!} \Bigl(\sum_{j=s-k}^\infty
\overline B(y+Na+ja)\Bigr)^{s-k}.
\end{eqnarray*}
Since the tail probability is non-increasing,
\begin{eqnarray*}
\sum_{j=s-k}^\infty\overline B(y+Na+ja)
\ge \frac{1}{a} \int_{(s-k)a}^\infty\overline B(y+Na+z)dz.
\end{eqnarray*}
Combining these expressions, we obtain the desired estimate
(\ref{Sigma}). Substituting (\ref{Sigma}) into
(\ref{union.An.N}) we get, as $y\to\infty$,
\begin{eqnarray}\label{Sigma.2}
\lim_{n\to\infty}
{\mathbb P}\Bigl\{\bigcup_{\vv i:i_{s-k}<n-N}
A_n(\vv i)\Bigr\}
&\ge& \frac{\rho^{s-k}-o(1)}{(s-k)!}
\overline{B_r}^{s-k}(y+(N+s-k)a).
\end{eqnarray}

Let $y=x(1+c\delta)$ in the definition of $A_n(\vv i)$.
Since an increment per unit of time of every
coordinate of the workload vector $\vv W_i$ is not
less than $-a$, on the event $A_n(\vv i)$
we have $W_{j,k+1}>x(1+c\delta)+(n-j)a$,
for all $j\in[i_{s-k}+1,n]$.
Consider $\vv i$ such that $i_{s-k}<n-N$.
Then $A_n(\vv i)$ implies
$W_{n-N,k+1}>x(1+c\delta)+Na$.
Together with (\ref{slln.double}) it yields
\begin{eqnarray}\label{cond.DA}
{\mathbb P}\Bigl\{D_n>x\Big|\bigcup_{\vv i:i_{s-k}<n-N}
A_n(\vv i)\Bigr\} &\to& 1
\end{eqnarray}
as $x\to\infty$, uniformly in $n$ and $N$. Now it
follows from (\ref{Sigma.2}) and (\ref{cond.DA}) that
\begin{eqnarray*}
{\mathbb P}\{D>x\}
&\ge& \liminf_{n\to\infty}
{\mathbb P}\Bigl\{W_{n1}>x\Big|
\bigcup_{\vv i:i_{s-k}<n-N}A_n(\vv i)\Bigr\}
{\mathbb P}\Bigl\{\bigcup_{\vv i:i_{s-k}<n-N}
A_n(\vv i)\Bigr\}\\
&\ge& \frac{\rho^{s-k}-o(1)}{(s-k)!}
\overline{B_r}^{s-k}(x(1+c\delta)+(N+s-k)a).
\end{eqnarray*}
Thus, by (\ref{def.N}),
\begin{eqnarray*}
{\mathbb P}\{D>x\}
&\ge& \frac{\rho^{s-k}-o(1)}{(s-k)!}
\overline{B_r}^{s-k}\Bigl(x\Bigl(
1+c\delta+\frac{ka}{b-ka-k\delta}\Bigr)+(s-k)a\Bigr).
\end{eqnarray*}
Hence, for every $\delta>0$,
\begin{eqnarray*}
{\mathbb P}\{D>x\}
&\ge& \frac{\rho^{s-k}+o(1)}{(s-k)!}
\overline{B_r}^{s-k} \Bigl(x
\Bigl(1+\frac{ka+\delta}{b-ka}\Bigr)\Bigr)\\
&=& \frac{\rho^{s-k}+o(1)}{(s-k)!}
\overline{B_r}^{s-k} \Bigl(x
\frac{b+\delta}{b-ka}\Bigr)
\ \mbox{ as }x\to\infty.
\end{eqnarray*}
The proof of Theorem \ref{th.lower} is complete.
\end{proof}

\section{New Majorant.}\label{majorant}

In the proof of the upper bound in Theorem \ref{co.s.max}
(see Section \ref{k=0}), we introduced $s$ parallel
single server queues that provide a suitable
majorant in the case $\rho<1$.
If $\rho\ge 1$ then the single server system with
service time distribution $B$ is unstable and
the above scheme does not work.
For an arbitrary $\rho$, we need a more
complex procedure to obtain a majorant.
Hereinafter let $k=[b/a]$ be the integer part of $b/a$.
We continue to assume constant interarrival times $\tau \equiv a$.

Again let $\sigma_{ni}$, $n\ge1$, $i\le s$, be
independent random variables with common distribution $B$.
Define service times $\sigma_n$ in the original
$D/GI/s$ system as in Section \ref{k=0}.
Consider again $s$ auxiliary single server queues $D/GI/1$,
but now with different deterministic arrival times
$T_n = n(k+1)(a-h)$ where
\begin{eqnarray}\label{choice.h}
\frac{k}{k+1}\Bigl(a-\frac{b}{k+1}\Bigr)<h<a-\frac{b}{k+1},
\end{eqnarray}
and with service times
equal to $\sigma_{ni}$ in queue $i=1,2,\ldots,s$. Then
$T_1=(k+1)(a-h)>b$, so that
%$U_{n1}$ has
%negative drift and the waiting times
the queues are stable.
%$U_{ni}$, $i=1$, \ldots, $s$, are stable.
Let $U_i$ be a stationary waiting time in
the $i$th auxiliary queue. Since, for each $n$,
\begin{eqnarray}\label{max_upper_independence}
\mbox{the sequences }
\{U_{n1},n\ge 1\},\ldots,\{U_{ns},n\ge 1\}
\mbox{ are mutually independent},
\end{eqnarray}
the limiting vector $(U_{1},\ldots,U_{s})$ consists
of independent identically distributed coordinates too.
%As ealier in Section \ref{k=0}, we define service
%times $\sigma_n$ in the original system
%by induction, by letting $\sigma_n=\sigma_{n,i_n}$.

In contrast to the case $\rho <1$,
it may not be true in general that, say, $V_{n1}$ is smaller than $U_{n1}$.
Nevertheless, for $\rho < k+1$, we can prove
that, for any set $I$ of $k+1$ indices,
$\sum_{i\in I} V_{ni}\le \sum_{i\in I}U_{ni}+\eta_I$
where $\eta_I$ has a light-tailed distribution,
this is Lemma \ref{l.sum.of.W.le.U} below.
But first we state the main result of the Section which is
an analogue of (\ref{D.min})
for the general $\rho$.

\begin{lemma}\label{l.D.le.orderU}
There exists a number $\beta>0$ and a random variable $\eta$ such that
${\mathbb E}e^{\beta\eta}<\infty$ and, for all $n$, with probability
1,
\begin{eqnarray*}
D_n &\le& U_{n,(k+1)}+\eta,
\end{eqnarray*}
where $U_{n,(k+1)}$ is the $(k+1)$th order
statistic of  vector $(U_{n1},\ldots,U_{ns})$.
\end{lemma}

Now we formulate and prove the following result.
Based on it, we give the proof of Lemma \ref{l.D.le.orderU}
at the end of the section.

\begin{lemma}\label{l.sum.of.W.le.U}
There exists $\beta>0$ such that, for any set of $k+1$
indices $I=\{i(1),\ldots,i(k+1)\}$, there is a random
variable $\eta_I$ such that
${\mathbb E}e^{\beta\eta_I}<\infty$ and, for any $n$,
with probability 1,
\begin{eqnarray*}
\sum_{i\in I} V_{ni} &\le&
\sum_{i\in I}U_{ni}+\eta_I.
\end{eqnarray*}
\end{lemma}

\begin{proof}
Fix some $i'\in I$.
Consider an auxiliary $GI/GI/k+1$ system
$\vv V'_n=(V'_{ni},i\in I)$ with the same interarrival
times equal to $\tau_n$, but whose service times $\sigma'_n$
are chosen in a special manner.
At any time $n$, if $i_n\in I$ ($i_n$ is defined in the proof
of Theorem \ref{co.s.max}, see Section \ref{k=0})
then put $\sigma'_n=\sigma_{n,i_n}$ and $i'_n=i_n$.
If $i_n\not\in I$ then put $\sigma'_n=\sigma_{ni'}$
and $i'_n=i'$.
Applying property (1) of Lemma \ref{monotonicity}, we
get that
%Again, in general, $i_n'\not=i_n$ and
%the vector $(V'_{ni}, i\in I)$
%does not majorise $(V_{ni},i\in I)$.
%Nevertheless,
$R(V_{ni},i\in I)\le R\vv V_n'$
coordinatewise, for any $n$. Therefore,
\begin{eqnarray*}
\sum_{i\in I} V_{ni}
&\le& \sum_{i\in I} V'_{ni}.
\end{eqnarray*}
Hence, it suffices to prove that
\begin{eqnarray}\label{reform.V}
\sum_{i\in I} V'_{ni} &\le&
\sum_{i\in I} U_{ni}+\eta_I.
\end{eqnarray}

For every $i\in I$ and for any $n$,
\begin{eqnarray*}
U_{n+1,i}-U_{ni}
&\ge& \sigma_{ni}-(k+1)(a-h),
\end{eqnarray*}
and hence
\begin{eqnarray}\label{increment.of.U}
\sum_{i\in I}U_{n+1,i}-\sum_{i\in I}U_{ni}
&\ge& \sigma_n'-(k+1)(a-h)+
\sum_{i\in I,i\not=i'_n}(\sigma_{in}-(k+1)(a-h))\nonumber\\
&=& \sigma_n'-(k+1)a+\zeta_{I,n},
\end{eqnarray}
where the independent identically distributed
random variables
$$
\zeta_{I,n}:=\sum_{i\in I,i\not=i'_n}(\sigma_{in}-b)
+(k+1)^2h+k(b-(k+1)a)
$$
have a positive mean, by the left inequality
%in the choice
in (\ref{choice.h}).
% of $h$.

Take $\varepsilon\in(0,{\mathbb E}\zeta_{I,n})$.
Let $d$ and $A$ be defined by Lemma \ref{slln.for.s-1}
applied to the system $(V'_{1+n,i},i\in I)$.
Consider $d$-skeleton of $\vv V'$, that is,
the sequence $\vv V'_{1+nd}$.
For every $n$, if the maximal coordinate of
$(V'_{1+nd,i},i\in I)$ is not bigger than $A$,
then
\begin{eqnarray*}
\sum_{i\in I}V'_{1+nd+d,i}-\sum_{i\in I}V'_{1+nd,i}
&\le& \sigma'_{nd+1}+\ldots+\sigma'_{nd+d},
\end{eqnarray*}
which together with (\ref{increment.of.U}) implies that
\begin{eqnarray*}
\sum_{i\in I}V'_{1+nd+d,i}-\sum_{i\in I}V'_{1+nd,i}
&\le& \sum_{i\in I}U_{1+nd+d,i}-\sum_{i\in I}U_{1+nd,i}
+d(k+1)a-\zeta_{I,1+nd}-\ldots-\zeta_{I,nd+d}\\
&\le& \sum_{i\in I}U_{1+nd+d,i}-\sum_{i\in I}U_{1+nd,i}
+d((k+1)a+b),
\end{eqnarray*}
since $\zeta_{I,n}\ge -b$. To conclude, if
the maximal coordinate of $(V'_{1+nd,i},i\in I)$
is not bigger than $A$, then
\begin{eqnarray*}
\sum_{i\in I}V'_{1+nd,i} &\le& (k+1)A,
\end{eqnarray*}
so that
\begin{eqnarray}\label{increment.of.Vprime.1}
\sum_{i\in I}V'_{1+nd+d,i}
&\le& \sum_{i\in I}U_{1+nd+d,i}+d((k+1)a+b)+(k+1)A\nonumber\\
&:=& \sum_{i\in I}U_{1+nd+d,i}+C,
\end{eqnarray}

Further, if the maximal coordinate of
$(V'_{1+nd,i},i\in I)$ is bigger than $A$,
then, by Lemma \ref{slln.for.s-1}, we have
\begin{eqnarray*}
\sum_{i\in I}V'_{1+nd+d,i}-\sum_{i\in I}V'_{1+nd,i}
&\le& \theta_{I,1+nd},
\end{eqnarray*}
where $\theta_{I,1+nd}$ are independent identically
distributed random variables with mean
${\mathbb E}\theta_{I,1+nd}\le d(b-(k+1)a+\varepsilon)$
and such that
\begin{eqnarray*}
\theta_{I,1+nd} &\le& \sigma'_{nd+1}+\ldots+\sigma'_{nd+d}.
\end{eqnarray*}
In this case, by (\ref{increment.of.U}),
\begin{eqnarray}\label{increment.of.Vprime.2}
\sum_{i\in I}V'_{1+nd+d,i}-\sum_{i\in I}V'_{1+nd,i}
&\le& \sum_{i\in I}U_{1+nd+d,i}-\sum_{i\in I}U_{1+nd,i}
+\theta_{I,1+nd}-(\sigma'_{nd}+\ldots+\sigma'_{nd+d})\nonumber\\
&&\hspace{20mm}+d(k+1)a-\zeta_{I,1+nd}-\ldots-\zeta_{I,nd+d}
\nonumber\\
&:=& \sum_{i\in I}U_{1+nd+d,i}-\sum_{i\in I}U_{1+nd,i}
+\widetilde\theta_{I,1+nd}.
\end{eqnarray}

Inequalities (\ref{increment.of.Vprime.1}) and
(\ref{increment.of.Vprime.2}) imply that always
\begin{eqnarray}\label{diff.rough}
\sum_{i\in I} V'_{1+nd+d,i}
&\le& \sum_{i\in I} U_{1+nd+d,i}+C
+\max_n\sum_{j=0}^n\widetilde\theta_{I,1+jd}.
\end{eqnarray}
Here $\widetilde{\theta}_{I,1+nd}, n=1,2,\ldots$ are
independent identically distributed random variables that
are bounded from above and have a negative mean. Indeed,
\begin{eqnarray*}
\widetilde\theta_{I,1+nd}
&=& \theta_{I,1+nd}-(\sigma'_{nd+1}+\ldots+\sigma'_{nd+d})
+d(k+1)a-\zeta_{I,1+nd}-\ldots-\zeta_{I,nd+d}\\
&\le& d(k+1)a-\zeta_{I,1+nd}-\ldots-\zeta_{I,nd+d}\\
&\le& d((k+1)a+b)
\end{eqnarray*}
and
\begin{eqnarray*}
{\mathbb E}\widetilde\theta_{I,1+nd}
&=& {\mathbb E}\theta_{I,1+nd}-bd
+d(k+1)a-d{\mathbb E}\zeta_{I,1+nd}\\
&\le& d(b-(k+1)a+\varepsilon)-bd+d(k+1)a-d{\mathbb E}\zeta_{I,1+nd}\\
&=& d(\varepsilon-{\mathbb E}\zeta_{I,1+nd})<0,
\end{eqnarray*}
by the choice of $\varepsilon$.
Therefore, there exists $\beta>0$ such that
${\mathbb E}e^{\beta\widetilde\theta_{I,1+nd}}<1$
and then the following estimate holds:
$$
{\mathbb E}\exp\Bigl\{\beta
\max_n\sum_{j=0}^n\widetilde\theta_{I,1+jd}\Bigr\}
\le \sum_n
({\mathbb E}e^{\beta\widetilde\theta_{I,1+nd}})^n
<\infty.
$$
Let
$$
\widetilde\eta_I:=C+\max_n\sum_{j=0}^n\widetilde\theta_{I,1+jd},
$$
then ${\mathbb E}e^{\beta\widetilde\eta_I}<\infty$ and,
by the upper bound (\ref{diff.rough}), for all $n$,
\begin{eqnarray}\label{V.U.int}
\sum_{i\in I} V'_{1+nd+d,i}
&\le& \sum_{i\in I} U_{1+nd+d,i}+\widetilde\eta_I.
\end{eqnarray}
Since, in addition, for every $l\in[1,d]$,
$$
\sum_{i\in I} V'_{1+nd+d+l,i}-\sum_{i\in I} V'_{1+nd+d+l-1,i}
\le \sum_{i\in I} U_{1+nd+d+l,i}-\sum_{i\in I} U_{1+nd+d+l-1,i}+(k+1)a,
$$
we conclude from (\ref{V.U.int}) that, for every $l\in[1,d]$,
\begin{eqnarray*}
\sum_{i\in I} V'_{1+nd+d+l,i}
&\le& \sum_{i\in I} U_{1+nd+d+l,i}+\widetilde\eta_I+l(k+1)a,
\end{eqnarray*}
so that, for every $n$,
\begin{eqnarray*}
\sum_{i\in I} V'_{ni}
&\le& \sum_{i\in I} U_{ni}+\widetilde\eta_I+d(k+1)a,
\end{eqnarray*}
which completes the proof of Lemma \ref{l.sum.of.W.le.U}
with $\eta_I:=\widetilde\eta_I+d(k+1)a$.
\end{proof}

\begin{proofof}{Lemma \ref{l.D.le.orderU}}
For every collection $I$ of $k+1$ coordinates
\begin{eqnarray*}
D_n &\le& \frac{1}{k+1}\sum_{i\in I} V_{ni},
\end{eqnarray*}
since $D_n$ is the minimal coordinate.
Then it follows from Lemma \ref{l.sum.of.W.le.U} that
\begin{eqnarray}\label{dn.le.sum}
D_n &\le& \frac{1}{k+1}\sum_{i\in I} U_{ni}+\eta_I.
\end{eqnarray}
Take $\eta:=\displaystyle\max_{I:|I|=k+1}\eta_I$.
Then ${\mathbb E}e^{\beta\eta}<\infty$ and
\begin{eqnarray}\label{dn.le.sum1}
D_n &\le& \frac{1}{k+1}\sum_{i\in I} U_{ni}+\eta.
\end{eqnarray}
Take $I$ such that $\{U_{ni},i\in I\}$
are the $k+1$ smallest coordinates of vector $(U_{n1},\ldots,U_{ns})$.
Then $U_{ni}\le U_{n,(k+1)}$ for every $i\in I$.
Together with (\ref{dn.le.sum1})
it yields the inequality of the lemma.
\end{proofof}

\section{Upper Bound and the Principle of $s-k$ Big Jumps.}\label{sec.upper.bound}

Now we turn to the upper bound.
Lemma \ref{l.D.le.orderU}
allows to prove the following general result.

\begin{theorem}\label{th.deter.upper}
Let $\rho<k+1$ for some $k\in\{0,1,\ldots,s-1\}$.
Then, for any fixed $h$ satisfying {\rm(\ref{choice.h})},
there exists $\beta>0$ such that
\begin{eqnarray*}
{\mathbb P}\{D>x+y\}
&\le& \Big(\begin{array}{c}s\\k\end{array}\Big)
(\overline F(x))^{s-k}+const\cdot e^{-\beta y}
\ \mbox{ for all }x,y>0,
\end{eqnarray*}
where $F$ is the distribution of random variable
$$
M:=\sup\Bigl(0,\ \sum_{j=1}^n
(\sigma_j-(k+1)(a-h)),\ n\ge 1\Bigr).
$$
\end{theorem}

\begin{proof}
First, by inequality
$b\equiv{\mathbb E}\sigma<(k+1)(a-h)$
and by the strong law of large numbers,
the maximum $M$ is finite with probability 1.
By Lemma \ref{l.D.le.orderU},
\begin{eqnarray*}
{\mathbb P}\{D_n>x+y\}
&\le& {\mathbb P}\{U_{n,(k+1)}+\eta>x+y\}\\
&\le& {\mathbb P}\{U_{n,(k+1)}>x\}+
{\mathbb P}\{\eta>y\}.
\end{eqnarray*}
Taking into account the independence of the $U$'s in
(\ref{max_upper_independence}), we obtain
\begin{eqnarray*}
{\mathbb P}\{D_n>x+y\}
&\le& \Big(\begin{array}{c}s\\k\end{array}\Big)
{\mathbb P}^{s-k}\{U_{n1}>x\}+
{\mathbb P}\{\eta>y\}.
\end{eqnarray*}
Letting $n\to\infty$ and taking into account
the duality between the single server system
and the maximum of the corresponding random walk,
we arrive at the following inequality:
\begin{eqnarray}\label{Dprime}
{\mathbb P}\{D>x+y\}
&\le& \Big(\begin{array}{c}s\\k\end{array}\Big)
(\overline F(x))^{s-k}+
{\mathbb P}\{\eta>y\}.
\end{eqnarray}
Since $\eta$ has a finite exponential
moment, we obtain the statement of the theorem.
\end{proof}

\begin{proofof}{the upper bound in Theorem \ref{th:subexp.bounds}}
It follows from Theorem \ref{th.deter.upper} that
\begin{eqnarray*}
{\mathbb P}\{D>x\}
&\le& \Big(\begin{array}{c}s\\k\end{array}\Big)
\overline F^{s-k}(x(1-\delta))
+const\cdot e^{-\beta\delta x}.
\end{eqnarray*}
Due to the subexponentiality of $B_r$ we obtain
from the analogue of (\ref{W.single}) for the maximum of
a random walk (see, e.g., \cite[Theorem 5.2]{FKZ})
that, as $x\to\infty$,
$$
\overline F(x)\sim\frac{b}{(k+1)(a-h)-b}\overline B_r(x).
$$
Taking $h$ in (\ref{choice.h}) close to its minimal value,
say $h=\frac{k}{k+1}\Bigl(a-\frac{b}{k+1}\Big)+\varepsilon$,
$\varepsilon>0$, we arrive at the following estimate:
$$
\overline F(x)\sim\frac{b}{a-b/(k+1)-(k+1)\varepsilon}\overline B_r(x)
=\frac{(k+1)\rho}{k+1-\rho-(k+1)^2\varepsilon/a}\overline B_r(x).
$$
In addition,
$\overline B_r(x(1-\delta))\cdot e^{\beta\delta x}
\to \infty$ as $x\to\infty$. All these facts and
arbitrarity of choice of $\varepsilon>0$ imply
the desired bound.
\end{proofof}

It what follows, for two families of events $A_x$ and $B_x$ of positive probabilities
indexed by $x$, we write $A_x \sim B_x$  if ${\mathbb P} \{A_x\setminus B_x\} = o({\mathbb P}\{A_x\})$
and ${\mathbb P} \{B_x\setminus A_x\} = o({\mathbb P}\{A_x\})$ as $x\to\infty$. Note that
$A_x\sim B_x$ implies ${\mathbb P}\{A_x\}\sim {\mathbb P}\{B_x\}$, but not vice versa.

We establish now the principle of $s-k$ big jumps in the case of
%SF
intermediate
%SF
regularly varying
distributions.
For simplicity, we do it again for $D/GI/s$ system with deterministic
inter-arrival times.
For this, we consider the representation of the stationary workload
in the backward time (the so-called ``Loynes scheme'').
We again use the joint representation of $s$ individual queues and of
the $s$-server system given in the previous section and assume that
all queues run for a long time, from time $-\infty$, and that
$U_i$ is the stationary waiting time of the customer that arrives at the $i$th
queue at time $0$. Then
$$
U_i = \sup (0, \xi_{-1,i}, \xi_{-1,i}+\xi_{-2,i}, \ldots ,
\xi_{-1,i}+\ldots +\xi_{-n,i}, \ldots )
$$
where $\xi_{j,i} =\sigma_{j,i}-\widehat{a}$, for $i=1,\ldots,s$ and $j=-1,-2,\ldots$,
and $\widehat{a}=(k+1)(a-h)$.
Further, ${\vv W}_n$ are stationary vectors for $n\le 0$, and
$$
{\vv W}_{n+1}= R((W_{n1}+\sigma_n-a)^+, (W_{n2}-a)^+,\ldots,(W_{n,s}-a)^+),
$$
for all $n<0$. Here again $i_n = \min \{i : \ V_{ni}=D_n \}$ and $\sigma_n=\sigma_{n,i_n}$.
Then, for any $x>0$, by Lemma \ref{l.D.le.orderU},
$$
\{D_0 >x\} \subset \bigcup_J \{ \min_{i\in J} U_i + \eta >x \}
$$
where $D_0$ is the stationary waiting time in the $D/GI/s$ queue,
i.e. the minimal coordinate of vector ${\vv W}_0$,
$J$'s are subsets of $\{1,2,\ldots,s\}$ of cardinality $s-k$, and $\eta$ is
a random variable with light-tailed distribution. Then
$$
\{D_0 > x\} = \bigcup_J \{D_0>x, \min_{i\in J} U_i + \eta >x \}.
$$

Assume that the
%SF
residual
%SF
distribution function $B_r$ of service times is
%SF
intermediate
%SF
regularly varying
%SF
(for that, it is sufficient for $B$ to be intermediate varying).
%$\overline{B}(x)= l(x)/x^{\alpha}$, $\alpha >1$, so
%$\overline{B}_r(x) \sim l(x)/(\alpha -1)x^{\alpha -1}$ and then
%${\mathbb P} (U_i >x)\sim cl(x)/(\alpha -1)x^{\alpha -1}$, for some constant $c$.
Then, clearly, each random variable $U_i$ has an intermediate regularly varying
distribution since ${\mathbb P} (U_i>x) \sim c \overline{B}_r(x)$.
Since
the random variables $U_i$ are mutually independent, the distribution
of $\min_{i\in J} U_i$ is also intermediate regularly varying,
${\mathbb P} (\min U_i >x) \sim c^{s-k} \left(\overline{B}_r(x)
\right)^{s-k}$.
% function which is proportional
%to $ \left(l(x)/x^{\alpha -1}\right)^{s-k}$.

It is well-known, see e.g. \cite[Ch 5]{FKZ}, that
$$
\{ U_i>x \} \sim \bigcup_{n\ge 1} \{\sigma_{-n,i}>x+n\widehat{a} \}
$$
and therefore, for any set $J \subset \{1,2,\ldots,s\}$,
$$
\{\min_{i\in J} U_i >x\} \sim  \bigcap_{i\in J} \left(\bigcup_{n>0}
\{\sigma_{-n,i}>x+n\widehat{a} \}\right).
$$

We use the following property of intermediate regularly varying distributions
(its proof is postponed until the end of the section;
a similar result for equivalence of probabilities
may be found in \cite{AAK}):

\begin{lemma}\label{irv3}
If $X$ and $Y$ are two random variables such that $X$
has an intermediate regularly varying distribution and
${\mathbb P}\{|Y|>x\}=o({\mathbb P}\{X>x\})$ as
$x\to\infty$, then $\{X+Y>x\} \sim \{X>x\}$,
for any joint distribution of $X$ and $Y$.
\end{lemma}

Applying Lemma \ref{irv3} with $X=\min_{i\in J} U_i$
and $Y=\eta$, for all $J$ of cardinality $s-k$, we get
\begin{eqnarray*}
\{D_0>x\}
&\sim &
\bigcup_J \{D_0>x, \min_{i\in J} U_i >x\}\\
&\sim &
\bigcup_J  \bigcap_{i\in J} \left(\bigcup_{n>0}
\{D_0>x, \sigma_{-n,i}>x+n\widehat{a} \}\right),
\end{eqnarray*}
since the upper and the lower bounds for
${\mathbb P}\{D_0>x\}$ are of the same order,
see \cite[Theorem 7]{BaF} or \cite{FK} for further arguments.
Now represent any event on the right in the latter equation as a union
of two events
$$
\{D_0>x,  \sigma_{-n,i}>x+n\widehat{a} , i_{-n}=i\} \cup
\{D_0>x,  \sigma_{-n,i}>x+n\widehat{a} , i_{-n}\neq i\}
$$
where
\begin{eqnarray*}
{\mathbb P}\{D_0>x, \sigma_{-n,i}>x+n\widehat{a}, i_{-n}\neq i\}
&=& {\mathbb P}\{D_0>x, i_{-n}\neq i\} {\mathbb P}\{\sigma_{-n,i}>x+n\widehat{a}\}\\
&\le& {\mathbb P}\{D_0>x\}{\mathbb P}\{\sigma_{-n,i}>x+n\widehat{a}\}.
\end{eqnarray*}
So, for any set $J$, the union of events
$$
\bigcap_{i\in J} \left(\bigcup_{n>0}
\{D_0>x, \sigma_{-n,i}>x+n\widehat{a}, i\neq i_{-n}\}\right)
$$
has probability $O({\mathbb P} \{D_0>x\} \overline{B}_r (x))
= o({\mathbb P} \{D_0>x\})$.
Since there is only a finite number of sets $J$,
we obtain the following result.

\begin{theorem}\label{bigjumps}
Assume that $\rho \in (k,k+1)$ and that the distribution of service
times
is intermediate regularly varying.
As $x\to\infty$,
\begin{equation}\label{bigj1}
{\mathbb P} \{D_0>x\} \sim {\mathbb P} \Bigl\{ D_0>x,
\bigcup_{0<n_1<n_2<\ldots <n_{s-k}}
\bigcap_{j=1}^{s-k} \{ \sigma_{-n_j}> x+n_j \widehat{a} \}
\Bigr\}.
\end{equation}
\end{theorem}

{\bf Remark.}
In the proof of Theorem \ref{bigjumps},
we followed the scheme introduced in \cite{FK},
see also \cite{BaF}, \cite{BMZ}, and \cite{BZ}
for similar constructions.
Theorem \ref{bigjumps} is not the final statement.
We may go further and obtain the following result.
%under the conditions of Theorem \ref{bigjumps},
Assume that $B$ is a regularly varying distribution.
Then,
for some positive and finite constant $C$ and as $x\to\infty$,
\begin{equation}\label{equivalence1}
{\mathbb P} \{D_0>x\} \sim C \overline{B}^{s-k}_{r}(x).
\end{equation}
The result seems to be correct, but its proof would be very lengthy and
would require a
scrupulous calculation, so we decided not to proceed further in this
direction.

We provide a hint for a plausible proof
instead. First, one may consider an auxiliary
deterministic model with  $(n-s)$ very big service times
$y_1$, \ldots, $y_{s-k}$ that occur
at time instants $-n_1>-n_2>\ldots >-n_{s-k}$ and replace all
other service times by their mean $b$. We also assume that, before
the first jump, the workload vector is zero.
For this model, we may find conditions on the $y$'s for the
minimal coordinate of the workload vector at time 0 to be not
smaller than $x$. Then repeat the same for all the other times
of jumps $-n_{1}>-n_{2}>\ldots >-n_{s-k}$. The union of these
regions may be represented as a combination of unions and differences
of a finite number of truncated half-spaces of  dimension $s-k$.
Summation of tail probabilities over each such set gives the
probability of order  $\overline{B}^{s-k}_{r}(x)$, thanks to the properties
of regularly varying functions. So a finite combination of sums and
differences of these probabilities gives a probability of the same
order. It cannot be of a lower order, due to the lower
bound.

\begin{proofof}{Lemma \ref{irv3}}
From Theorem 2.47 in \cite{FKZ}, if $X$ has
an intermediate regularly varying distribution, then
$$
{\mathbb P}\{X>x+h(x)\}\sim {\mathbb P}\{X>x\}
\sim {\mathbb P}\{X>x-h(x)\}
$$
as $x\to\infty$, for any function $h(x)\to\infty$ such
that $h(x)=o(x)$. Hence, by the monotonicity arguments,
$$
\{X>x+h(x)\}\sim \{X>x\}\sim \{X>x-h(x)\}.
$$
Since the distribution of $X$ is intermediate regularly
varying and since ${\mathbb P}\{|Y|>x\}=o({\mathbb P}\{X>x\})$
as $x\to\infty$, we have 
${\mathbb P}\{|Y|>\varepsilon x\}=o({\mathbb P}\{X>x\})$
as $x\to\infty$, for every $\varepsilon>0$. Then 
there exists $h(x)=o(x)$ such that
$$
{\mathbb P}\{|Y|>h(x)\}=o({\mathbb P}\{X>x\}).
$$
Therefore, as $x\to\infty$,
\begin{eqnarray*}
\{X>x\} &\sim& \{X>x+h(x)\}\setminus\{Y<-h(x)\}\\
&=& \{X>x+h(x),Y\ge-h(x)\} \subseteq \{X+Y>x\}
\end{eqnarray*}
and
\begin{eqnarray*}
\{X>x\} &\sim& \{X>x-h(x)\}\cup\{Y>h(x)\}
\supseteq \{X+Y>x\}, 
\end{eqnarray*}
which justifies the events equivalence,
$\{X+Y>x\} \sim \{X>x\}$.
\end{proofof}

\section{Existence of moments: proof of Theorem 4.}\label{proof}

Since the tail distribution of $\min(\sigma_{r,1},\ldots,\sigma_{r,s-k})$
is equal to $(\overline B_r(x))^{s-k}$, we obtain from
Theorem \ref{th.lower}
$$
{\mathbb P}\{D>x\} \ge
c_1{\mathbb P}\{\min(\sigma_{r,1},\ldots,\sigma_{r,s-k})
>c_2x\}.
$$
Since, for any non-negative random variable $\eta$,
$$
{\mathbb E}\eta^\gamma = \gamma\int_0^\infty x^{\gamma-1}
{\mathbb P}\{\eta>x\}dx,
$$
we have
$$
{\mathbb E}D^\gamma \ge \frac{c_1}{c_2}
{\mathbb E}(\min(\sigma_{r,1},\ldots,\sigma_{r,s-k}))^\gamma.
$$
and the existence of the moment of order $\gamma$
for the delay $D$ implies with necessity (\ref{min.fin}).

Now assume (\ref{min.fin}). Consider $s-k$
independent copies $M_1$, \ldots, $M_{s-k}$ of the
random variable $M$ introduced in Theorem \ref{th.deter.upper}.
Then the assertion of Theorem \ref{th.deter.upper}
can be rewritten in the following way:
$$
{\mathbb P}\{D>x+y\} \le
\Big(\begin{array}{c}s\\k\end{array}\Big)
{\mathbb P}\{\min(M_1,\ldots,M_{s-k})>x\}
+const\cdot e^{-\beta y}.
$$
Take $y=x$. Then ${\mathbb E}D^\gamma<\infty$ follows
if we prove that
\begin{equation}\label{M1.Ms-k}
{\mathbb E}(\min(M_1,\ldots,M_{s-k}))^\gamma < \infty.
\end{equation}
In order to do it, we explore the ladder height construction
for the maximum $M$ of a random walk $S_n=X_1+\ldots+X_n$
where $X_j=\sigma_j-b-\varepsilon$.
Since this random walk has a negative drift,
the first ladder epoch and the first ladder height
$$
\theta=\min(n\ge 1:S_n>0), \quad \widetilde\chi=S_{\theta},
$$
both are degenerate random variables;
$$
p\equiv {\mathbb P}\{\theta<\infty\}={\mathbb P}\{M>0\}<1.
$$
Denote by $\chi$ a random variable with distribution
$$
{\mathbb P}\{\chi\in B\}={\mathbb P}\{\widetilde\chi\in B\}/p.
$$
Let $\chi_j$ be independent copies of $\chi$.
If $\eta$ is an independent counting random variable with
distribution ${\mathbb P}\{\eta=j\}=(1-p)p^j$, $j=0$, $1$, \ldots,
then $M$ is equal in distribution to $\chi_1+\ldots+\chi_\eta$.

Let $\chi_{i,j}$ be again independent copies of $\chi$
and $\eta_j$ be independent copies of $\eta$.
Then $\min(M_1,\ldots,M_{s-k})$ is equal in distribution to
$$
\min\biggl(\sum_{j=1}^{\eta_1}\chi_{1,j},\ldots,
\sum_{j=1}^{\eta_{s-k}}\chi_{s-k,j}\biggr).
$$
%Since the function $\min$ is concave,
The latter
minimum does not exceed
$$
\sum_{j_1=1}^{\eta_1}\ldots\sum_{j_{s-k}=1}^{\eta_{s-k}}
\min(\chi_{1,j_1},\ldots,\chi_{s-k,j_{s-k}}).
$$
Taking into account that for non-negative arguments
$$
(x_1+\ldots+x_N)^\gamma\le N^\gamma(x_1^\gamma+\ldots+x_N^\gamma),
$$
we get the following estimate:
\begin{eqnarray*}
\min\biggl(\sum_{j=1}^{\eta_1}\chi_{1,j},\ldots,
\sum_{j=1}^{\eta_{s-k}}\chi_{s-k,j}\biggr)^\gamma
&\le& (\eta_1+\ldots+\eta_{s-k})^\gamma
\sum_{j_1=1}^{\eta_1}\ldots\sum_{j_{s-k}=1}^{\eta_{s-k}}
\min(\chi_{1,j_1},\ldots,\chi_{s-k,j_{s-k}})^\gamma.
\end{eqnarray*}
In particular, the mean of the term in the left side of the equality above
is not larger than
\begin{eqnarray*}
\lefteqn{\sum_{j_1=1}^\infty\ldots\sum_{j_{s-k}=1}^\infty
(1-p)p^{j_1+\ldots+j_{s-k}}(j_1+\ldots+j_{s-k})^\gamma
j_1\cdot\ldots\cdot j_{s-k}
{\mathbb E}\min(\chi_{1,1},\ldots,\chi_{s-k,1})^\gamma}\\
&&\hspace{40mm} = {\mathbb E}(\eta_1+\ldots+\eta_{s-k})^\gamma
\eta_1\ldots\eta_{s-k}
{\mathbb E}\min(\chi_{1,1},\ldots,\chi_{s-k,1})^\gamma.
\end{eqnarray*}
Since the $\eta$'s have finite exponential moments,
the first mean on the right is finite.
Now we show finiteness of the second mean. First,
%Indeed, the distribution
%of $\chi$ can be estimated in the following way: first,
$$
{\mathbb P}\{\chi>x\} =
\int_{-\infty}^0 \overline B(x-y)\mu(dy),
$$
where the measure $\mu$ is defined by
\begin{eqnarray*}
\mu(dy) &=& \sum_n {\mathbb P}\{S_n\in dy, S_k\le 0
\mbox{ for all }k\le n-1\}\\
&\le& \sum_n {\mathbb P}\{S_n\in dy\}.
\end{eqnarray*}
Then, by the key renewal theorem,
$$
c\equiv \sup_{y\le 0}\mu(y-1,y] < \infty,
$$
which yields
$$
{\mathbb P}\{\chi>x\} \le c \sum_{j=0}^\infty
\overline B(x+j)\le c\overline B_r(x-1).
$$
Therefore, due to condition (\ref{min.fin}),
$$
{\mathbb E}\min(\chi_{1,1},\ldots,\chi_{s-k,1})^\gamma<\infty,
$$
which completes the proof.

\acknowledgments{This research was supported
by EPSRC grant No.~R58765/01
and RFBR grant No.~10-01-00161.}
The authors thank Bert Zwart for stimulating discussions, 
and Arcady Shemyakin and James Cruise for stylistic comments.

\end{document}